\documentclass[a4paper,12pt]{amsart}
\usepackage{amssymb}
\usepackage{cite}
\usepackage{ifthen}
\usepackage[dvips]{graphicx}
\usepackage{float}
\usepackage{subcaption}
\nonstopmode \numberwithin{equation}{section}
\newtheorem*{theoA}{Theorem A}
\newtheorem*{theoB}{Theorem B}

\usepackage{amssymb}
\usepackage{ifthen}
\usepackage{graphicx}
\usepackage{amsmath}
\usepackage[T1]{fontenc} 
\usepackage[utf8]{inputenc}
\usepackage[usenames,dvipsnames]{color}
\usepackage{color}
\usepackage[english]{babel}
\usepackage{fancyhdr}
\usepackage{fancybox}
\usepackage{tikz}

\theoremstyle{plain}
\newtheorem{prop}{Proposition}

\newtheorem{conj}{Conjecture}

\theoremstyle{definition}
\newtheorem{defi}{Definition}[section]

\newtheorem{cor}{Corollary}[section]
\newtheorem{thm}{Theorem}[section]
\newtheorem{ques}{Question}[section]
\newtheorem{theo}{Theorem}[section]
\newtheorem{lem}{Lemma}[section]
\newtheorem{prob}{Problem}
\newtheorem{rem}{Remark}[section]

\theoremstyle{plain}

\newcounter{minutes}
\divide\time by 60
\newcounter{hours}
\multiply\time by 60
\addtocounter{minutes}{-\time}

\newcounter {own}
\def\theown {\thesection       .\arabic{own}}

\newenvironment{pf}[1][]{%
	\vskip 3mm
	\noindent
	\ifthenelse{\equal{#1}{}}%
	{{\slshape Proof. }}%
	{{\slshape #1.} }%
}%
{\qed\bigskip}

\newcounter{alphabet}

\def\be{\begin{equation}}
	\def\ee{\end{equation}}

\newcommand{\bee}{\begin{enumerate}}
	\newcommand{\eee}{\end{enumerate}}

\newcommand{\blem}{\begin{lem}}
	\newcommand{\elem}{\end{lem}}
\newcommand{\bthm}{\begin{thm}}
	\newcommand{\ethm}{\end{thm}}
\newcommand{\bcor}{\begin{cor}}
	\newcommand{\ecor}{\end{cor}}
\newcommand{\beg}{\begin{examp}}
	\newcommand{\eeg}{\end{examp}}
\newcommand{\begs}{\begin{examples}}
	\newcommand{\eegs}{\end{examples}}

\newcommand{\bdefn}{\begin{defn}}
	\newcommand{\edefn}{\end{defn}}

\newcommand{\bprob}{\begin{prob}}
	\newcommand{\eprob}{\end{prob}}
\newcommand{\bei}{\begin{itemize}}
	\newcommand{\eei}{\end{itemize}}

\newcommand{\bcon}{\begin{conj}}
	\newcommand{\econ}{\end{conj}}
\newcommand{\bcons}{\begin{conjs}}
	\newcommand{\econs}{\end{conjs}}
\newcommand{\bprop}{\begin{prop}}
	\newcommand{\eprop}{\end{prop}}
\newcommand{\br}{\begin{rem}}
	\newcommand{\er}{\end{rem}}
\newcommand{\brs}{\begin{rems}}
	\newcommand{\ers}{\end{rems}}
\newcommand{\bo}{\begin{obser}}
	\newcommand{\eo}{\end{obser}}
\newcommand{\bos}{\begin{obsers}}
	\newcommand{\eos}{\end{obsers}}
\newcommand{\bpf}{\begin{pf}}
	\newcommand{\epf}{\end{pf}}
\newcommand{\ba}{\begin{array}}
	\newcommand{\ea}{\end{array}}
\newcommand{\beq}{\begin{eqnarray}}
	\newcommand{\beqq}{\begin{eqnarray*}}
		\newcommand{\eeq}{\end{eqnarray}}
	\newcommand{\eeqq}{\end{eqnarray*}}

\begin{document}

\title{Geometric subfamily of locally univalent functions, Blaschke products and quasidisk}

\author{Molla Basir Ahamed$^*$}
\address{Molla Basir Ahamed, Department of Mathematics, Jadavpur University, Kolkata-700032, West Bengal, India.}
\email{mbahamed.math@jadavpuruniversity.in}

\author{Rajesh Hossain}
\address{Rajesh Hossain, Department of Mathematics, Jadavpur University, Kolkata-700032, West Bengal, India.}
\email{rajesh1998hossain@gmail.com}

\subjclass[{AMS} Subject Classification:]{Primary: 30C45, 30C62, 30C80, 30J10, 31C05, Secondary: 30C20, 30C55, 31A05}
\def\thefootnote{}
\footnotetext{ {\tiny File:~\jobname.tex,
printed: \number\year-\number\month-\number\day,
          \thehours.\ifnum\theminutes<10{0}\fi\theminutes }
} \makeatletter\def\thefootnote{\@arabic\c@footnote}\makeatother
\begin{abstract} 
We investigate the family $\mathcal{F}(\alpha)$, where $\alpha \in (0,3]$, of locally univalent analytic functions defined by the condition:$$\operatorname{Re}\left(1 + \frac{zf''(z)}{f'(z)}\right) > 1 - \frac{\alpha}{2}, \quad z \in \mathbb{D}.$$Our main contributions include establishing sharp bounds for the Schwarzian and pre-Schwarzian derivatives within this family and revealing their intrinsic connection with finite Blaschke products. Furthermore, we prove that the image domains $f(\mathbb{D})$ are quasidisks. Finally, we extend these results to sense-preserving harmonic mappings $f = h + \overline{g}$ whose analytic part $h$ belongs to $\mathcal{F}(\alpha)$, and we derive optimal pre-Schwarzian norm estimates for this class.
\end{abstract}
\maketitle
\pagestyle{myheadings}
\markboth{M. B. Ahamed and R. Hossain}{Geometric subfamily of locally univalent functions and Blaschke products, quasidisks}
\section{\bf Introduction}
For $r > 0$, let $\mathbb{D}_r := \{z \in \mathbb{C} : |z| < r \}$ and $\mathbb{D} := \mathbb{D}_1$ be the open unit disk. Let $\mathbb{T} := \partial\mathbb{D} = \{z \in \mathbb{C} : |z| = 1\}$ denote the unit circle. Let $\mathcal{A}$ be the family of analytic functions $h$ in $\mathbb{D}$ normalized by $h(0) = 0$ and $h'(0) = 1$, and let $\mathcal{S}$ be the subfamily of $\mathcal{A}$ consisting of all univalent functions in $\mathbb{D}$. We denote by $\mathcal{S}^*$ the subfamily of $\mathcal{S}$ consisting of starlike functions; i.e., functions for which $f(\mathbb{D})$ is a domain starlike with respect to the origin. Note that $\mathcal{S}$ is a complete metric space.\vspace{2mm}

Let $f$ and $g$ be two analytic functions in the unit disk $\mathbb{D}$. We say that $f$ is \emph{subordinate} to $g$, denoted by $f \prec g$, if there exists an analytic function $\omega$ with $\omega(0)=0$ and $|\omega(z)| < 1$ such that $f(z) = g(\omega(z))$ for $z \in \mathbb{D}$. In particular, if $g$ is univalent in $\mathbb{D}$, then $f \prec g$ if and only if $f(0)=g(0)$ and $f(\mathbb{D}) \subset g(\mathbb{D})$ (see \cite{Duren-1983}). Let $\mathcal{P}$ denote the class of analytic functions $p$ in $\mathbb{D}$ such that $p(0)=1$ and ${\rm Re}\, p(z) > 0$ for $z \in \mathbb{D}$. This class $\mathcal{P}$ is the well-known Carathéodory class. Furthermore, let $\phi \in \mathcal{P}$ be a univalent function such that $\phi'(0) > 0$, where the image domain $\phi(\mathbb{D})$ is starlike with respect to $1$ and symmetric with respect to the real axis.\vspace{1.2mm}

 The following unified families of convex and starlike functions have been 
\begin{align*}
	\mathcal{C}(\phi):=\bigg\{f\in\mathcal{S} : 1+\frac{zf''(z)}{f'(z)}\prec \phi(z)\; \mbox{for}\; z\in\mathbb{D}\bigg\}
\end{align*}
and 
\begin{align*}
	\mathcal{S}^*(\phi):=\bigg\{f\in\mathcal{S} : \frac{zf'(z)}{f(z)}\prec \phi(z)\; \mbox{for}\; z\in\mathbb{D}\bigg\},
\end{align*}
respectively, introduced and investigated by Ma and Minda \cite{Ma-Minda-1992}. In particular, when $\phi(z)=(1+z)/(1-z)$, then $\mathcal{C}(\phi)$ and $\mathcal{S}^*(\phi)$ are the families $\mathcal{C}$ and $\mathcal{S}^*$ of convex and starlike functions, respectively.\vspace{2mm}

For an appropriate choice of function $\phi$, one can obtain the subfamilies of $\mathcal{C}(\phi)$ and $\mathcal{S}^*(\phi)$. For instance,\vspace{2mm}

$\bullet$ $\mathcal{S}^*\left(\frac{1+Az}{1+Bz}\right)=:\mathcal{S}^*(A, B)$, the family of Janowski starlike functions.\vspace{2mm}

$\bullet$ $\mathcal{C}\left(\frac{1+(1-2\beta)z}{1-z}\right)=:\mathcal{C}(\beta)$, the family of convex functions of order $\beta\in [0, 1)$. \vspace{2mm}

$\bullet$ $\mathcal{C}\left(\frac{1+Az}{1+Bz}\right)=:\mathcal{C}(A, B)$, the family of Janowski convex functions.\vspace{2mm}

$\bullet$ $\mathcal{C}\left(1+\frac{\alpha z}{1-z}\right)=:\mathcal{F}(\alpha)$ for some $\alpha\in (0, 3]$, \emph{i.e.,} the family $\mathcal{F}(\alpha)$ is defined by 
\begin{align*}
	\mathcal{F}(\alpha)=\bigg\{h\in\mathcal{A} : {\rm Re}\left(1+\frac{zh''(z)}{h'(z)}\right)>1-\frac{\alpha}{2}\; \mbox{for}\; z\in\mathbb{D}\bigg\}.
\end{align*}
In particular, $\mathcal{F}(2)=:\mathcal{C}$ and $\mathcal{F}(3)=:\mathcal{F}(-1/2)$.\vspace{2mm}

These families have attracted the focus of many researchers in recent years due to their various geometric and analytic properties in geometric function theory. For further details on these subfamilies, we refer the reader to \cite{Ahamed-Hossain-2024, Allu-Sharma-BSM-2024} and the references therein.\vspace{2mm}

In the present article, we consider the family $\mathcal{F}(\alpha)$ of locally univalent functions $h \in\mathcal{A}$ such that
\begin{align}\label{Eq-1.1}
		{\rm Re}\left(1+\frac{zh''(z)}{h'(z)}\right) > 1 - \frac{\alpha}{2} \quad \text{for } z \in \mathbb{D}.
\end{align}
\indent  Our primary contribution in this paper lies in establishing sharp bounds for both the Schwarzian and pre-Schwarzian derivatives within this class, while demonstrating their fundamental relationship with finite Blaschke products. We further characterize the geometric properties of this family by proving that the image domains $f(\mathbb{D})$ constitute quasidisks, a significant result that ensures the boundary of the image is a Jordan curve with specific regularity. Finally, we broaden the scope of our study by extending these analytic findings to the realm of sense-preserving harmonic mappings $f = h + \overline{g}$; by assuming the analytic part $h$ belongs to $\mathcal{F}(\alpha)$, we successfully derive sharp pre-Schwarzian norm estimates for this more general class of functions.\vspace{1.2mm}

The article is organized as follows. In Section \ref{Sec-2}, we develop the necessary background to characterize functions in the class $\mathcal{F}(\alpha)$ in terms of finite Blaschke products (see Theorem \ref{Th-2.1}). In Section \ref{Sec-3}, we obtain a sharp estimate of the Schwarzian derivative for functions belonging to the family $\mathcal{F}(\alpha)$ (see Theorem \ref{Th-3.2}). In Section \ref{Sec-4}, we discuss the properties of quasidisks; specifically, we establish a pivotal lemma (Lemma \ref{Lem-4.1}) that serves as a cornerstone for the derivation of our primary results. In Section \ref{Sec-5}, we obtain results on univalent harmonic mappings whose analytic parts belong to $\mathcal{F}(\alpha)$. Finally, in Section \ref{Sec-6}, we study pre-Schwarzian derivatives for harmonic mappings with a fixed analytic part in $\mathcal{F}(\alpha)$.

\section{\bf Properties and characterization of functions in $\mathcal{F}(\alpha)$}\label{Sec-2}Let $\mathbb{B}$ denote the set of all analytic functions $\omega$ in $\mathbb{D}$ satisfying $|\omega(z)| \le 1$ for all $z \in \mathbb{D}$, and let $\mathbb{B}_0 = \{\omega \in \mathbb{B} : \omega(0) = 0\}$.\vspace{2mm}

We obtain the following sharp inequality for functions in the class $\mathcal{F}(\alpha)$.
\begin{thm}\label{Th-2.1}
	Let $h\in\mathcal{F}(\alpha)$ for some $\alpha\in(0,3]$. Then 
	\begin{align}\label{Eq-2.1}
		{\rm Re  }\left(\frac{zh^{\prime\prime}(z)}{h^{\prime}(z)}\right)\geq -\frac{\alpha}{2}+\left(\frac{1-|z|^2}{2\alpha}\right)\bigg|\frac{h^{\prime\prime}(z)}{h^{\prime}(z)}\bigg|^2\;\mbox{for all}\; z\in\mathbb{D}.
	\end{align}
	Sharp inequality holds for all $z\in\mathbb{D}$ unless $h'(z)=(1-z\zeta)^{-\alpha}$ for some $\zeta\in\mathbb{T}.$
\end{thm}
\begin{proof}
	By assumption, \eqref{Eq-1.1} holds  and thus, there exists a $\omega\in\mathbb{B}_0$, \emph{i.e.,} $\omega:\mathbb{D}\rightarrow\mathbb{D}$ with $\omega(0)=0$, such that
	\begin{align}\label{Eq-2.2A}
		\frac{h^{\prime\prime}(z)}{h^{\prime}(z)}=\frac{\alpha\omega(z)}{z(1-\omega(z))}\; \mbox{for all}\; z\in\mathbb{D}. 
	\end{align}
	Since $\omega(0) = 0$, we may set $\omega(z) = z\phi(z)$ so that \eqref{Eq-2.2A} reduces to
	\begin{align}\label{Eq-2.2}
		\frac{h^{\prime\prime}(z)}{h^{\prime}(z)}=\frac{\alpha\phi(z)}{(1-z\phi(z)},
	\end{align}
	where $\phi\in\mathbb{B}$. Rewriting \eqref{Eq-2.2} yields that
	\begin{align}\label{Eq-2.4A}
		\phi(z)=\frac{\frac{h^{\prime\prime}(z)}{h^{\prime}(z)}}{\frac{zh^{\prime\prime}(z)}{h^{\prime}(z)}+\alpha}\;\;\mbox{for}\;z\in\mathbb{D}.
	\end{align}
	Since $|\phi(z)| \le 1$, \eqref{Eq-2.4A} we have
	\begin{align*}
		\bigg|\frac{h^{\prime\prime}(z)}{h^{\prime}(z)}\bigg|^2\leq\bigg|\frac{zh^{\prime\prime}(z)}{h^{\prime}(z)}+\alpha\bigg|^2\;\mbox{for}\;z\in\mathbb{D}.
	\end{align*}
	A simple calculation shows that
	\begin{align*}
		\bigg|\frac{h^{\prime\prime}(z)}{h^{\prime}(z)}\bigg|^2\leq\bigg|\frac{zh^{\prime\prime}(z)}{h^{\prime}(z)}\bigg|^2+2\alpha{\rm Re  }\left(\frac{zh^{\prime\prime}(z)}{h^{\prime}(z)}\right)+\alpha^2,
	\end{align*}
	which equivalent to
	\begin{align*}
		{\rm Re  }\left(\frac{zh^{\prime\prime}(z)}{h^{\prime}(z)}\right)\geq -\frac{\alpha}{2}+\left(\frac{1-|z|^2}{2\alpha}\right)\bigg|\frac{h^{\prime\prime}(z)}{h^{\prime}(z)}\bigg|^2\;\mbox{for all}\; z\in\mathbb{D}.
	\end{align*}
	
	 If the equality in \eqref{Eq-2.1} holds for some point $z_0\in\mathbb{D}$, then $|\phi(z_0)|=1$ and hence, $\phi(z)\equiv\zeta$ for some $\zeta\in\mathbb{T}$. From the euation \eqref{Eq-2.2}
	\begin{align*}
		\frac{h^{\prime\prime}(z)}{h^{\prime}(z)}=\frac{\alpha\zeta}{1-z\zeta}
	\end{align*}
	which by integration shows that $h^{\prime}(z)=(1-z\zeta)^{-\alpha}$ as desired.
\end{proof}
\section{\bf Characterization of functions from the class $\mathcal{F}(\alpha)$}\label{Sec-3}
We obtain the following result a characterization for functions which gives in the class $\mathcal{F}(\alpha)$.
\begin{thm}\label{Th-3.1}
	Let $h\in\mathcal{F}(\alpha)$ for some $\alpha\in(0,3]$. Then 
	\begin{align}\label{Eq-3.1}
		h^{\prime}(z)=\exp\;\left(-\alpha\int_{\mathbb{T}}\log(1-z\zeta)d\mu(\zeta)\right)\;\mbox{for}\;z\in\mathbb{D}
	\end{align}
	where $\mu$ is a probability measure on $\mathbb{T}$ so that so that $\int_{\mathbb{T}}d\mu(\zeta)=1$. Further, we have $h^{\prime}(z)\prec H_{\alpha}(z)$ for $z\in\mathbb{D}$, where $H_{\alpha}(z)=(1-z)^{-\alpha}$.
\end{thm}
\begin{rem}
	Theorem \ref{Th-3.1} establishes that if $h \in \mathcal{F}(\alpha)$, then $h'$ is subordinate to $H_{\alpha}(z) = (1-z)^{-\alpha}$. This subordination ensures that the image $f(\mathbb{D})$ is a quasidisk, a stronger form of univalence that allows for quasiconformal extension to the whole complex plane $\mathbb{C}$.
\end{rem}
\begin{proof}
	Let $h\in\mathcal{F}(\alpha)$. Then, from the analytic characterization for functions in the class $\mathcal{F}(\alpha)$, we have an equivalent condition
	\begin{align*}
		{\rm Re  }\left(1+\frac{2}{\alpha}\frac{zh^{\prime\prime}(z)}{h^{\prime}(z)}\right)>0\;\mbox{for}\;z\in\mathbb{D},
	\end{align*}
	and hence, by the Herglotz representation for analytic functions with positive real part in the unit disk, it follows easily that
	\begin{align*}
		1-\frac{2}{\alpha}\frac{zh^{\prime\prime}(z)}{h^{\prime}(z)}=\int_{\mathbb{T}}\frac{1+z\zeta}{1-z\zeta}d\mu(\zeta),
	\end{align*}
	which is equivalent to
	\begin{align}\label{Eq-3.2}
		\frac{zh^{\prime\prime}(z)}{h^{\prime}(z)}=\alpha\int_{\mathbb{T}}\frac{\zeta}{1-z\zeta}d\mu(\zeta)\;\mbox{for}\;z\in\mathbb{D}
	\end{align}
	where $\mu$ is unit probability measure on $\mathbb{T}$, \emph{i.e.,}$\int_{\mathbb{T}}d\mu(\zeta)=1$. An integrating \eqref{Eq-3.2}, we obtain
	\begin{align*}
	\log h^{\prime}(z)=-\alpha\int_{\mathbb{T}}\log(1-z\zeta)d\mu(\zeta)
	\end{align*}
	and the desired conclusion \eqref{Eq-3.1} follows if we perform exponentiation on both sides of this relation.\vspace{2mm}
	
	The next fact that $h^{\prime}(z)\prec H_{\alpha}(z)=(1-z)^{-\alpha}$ is well-known and is used for example in \cite{Allu-Sharma-BSM-2024}. Indeed if $h\in\mathcal{F}(\alpha)$ for some $\alpha\in(0,3]$, then this is equivalent to 
	\begin{align*}
		z\frac{d}{dz}(\log p(z))=	z\frac{p'(z)}{p(z)}\prec\frac{\alpha z}{1-z}=\frac{zH^{\prime}_{\alpha}(z)}{H_{\alpha}(z)},
	\end{align*} 
	where $p(z)=h^{\prime}(z)$ and $H_\alpha(z)=(1-z)^{-\alpha}$. But the last subordination gives $h^{\prime}(z)\prec H_{\alpha}(z)=(1-z)^{-\alpha}$, by the definition of subordination. This completes the proof.
\end{proof}
\begin{thm}\label{Th-3.2}
	Suppose that $f \in \mathcal{F}(\alpha)$ for some $\alpha\in (0, 3]$, and satisfies \eqref{Eq-2.2A} for some $\phi \in \mathcal{B}$. Then $\varphi$ is a finite Blaschke product with degree $m \geq 1$ if, and only if,
	\begin{align*}
		f^{\prime}(z)=\prod_{k=1}^{m+1}\frac{1}{(1-\zeta_k z)^{\alpha t_k}},
	\end{align*}
	where $\zeta_k \in \mathbb{T}$ are distinct points, $0 < t_k < 1$ and $\sum_{k=1}^{m+1} t_k = 1$.
\end{thm}
\begin{rem}
	The condition ${\rm Re}(1 + \frac{zh''(z)}{h'(z)}) > 1 - \frac{\alpha}{2}$ for $h \in \mathcal{F}(\alpha)$ directly relates to the classical conditions for convexity and close-to-convexity. For instance, $\mathcal{F}(2)$ corresponds to the class of convex functions $\mathcal{C}$, which are a subset of close-to-convex functions. The Blaschke characterization ensures the analytic part maintains the necessary boundary behavior for the harmonic map $f$ to remain close-to-convex (\textit{i.e.}, its image complement is a union of non-intersecting half-lines).
\end{rem}
\begin{proof}
	First, we observe that the family $\mathcal{F}(\alpha)$ is rotationally invariant in the sense that if $f \in \mathcal{F}(\alpha)$ then for each $\theta \in [0, 2\pi)$, $e^{-i\theta} f(e^{i\theta} z)$ belongs to $\mathcal{F}(\alpha)$. Therefore, if $\phi$ is a finite Blaschke product with degree $m \geq 1$, then by rotating $f$, we may assume that
	\begin{align*}
		\phi(z)=\prod_{k=1}^{m}\frac{z-b_k}{1-\overline{b_k}z},\;\;b_k\in\mathbb{D}.
	\end{align*}
	Therefore, we have
	 \begin{align}\label{Eq-3.33}
	 	\frac{\phi(z)}{z\phi(z)-1}=\frac{\prod_{k=1}^{m}(z-b_k)}{z\prod_{k=1}^{m}(z-b_k)-\prod_{k=1}^{m}(1-\overline{b_k}z)}
	 \end{align}
	From the right side of \eqref{Eq-3.33}, we can see that $\frac{\phi(z)}{z\phi(z)-1}$ is a rational function with poles at the roots of $z\phi(z) = 1$. Since $z\varphi(z)$ is a finite Blaschke product with degree $m+1$, Theorem 3.4.10 in \cite{Garcia-Mashreghi-Ross-2018} implies that there exist $m+1$ distinct roots $z_1, z_2, \ldots, z_{m+1}$ on $\mathbb{T}$ of $z\phi(z) = 1$. These points are simple poles of the meromorphic function $\frac{\phi(z)}{z\phi(z)-1}$. A partial fraction expansion of this function thus gives that
\begin{align}\label{Eq-3.3}
	\frac{\phi(z)}{z\phi(z)-1}=\sum_{k=1}^{m+1}\frac{t_k}{z-z_k},
\end{align}
where $t_k \neq 0$ are complex constants. Since the right side of \eqref{Eq-3.33}is the quotient of two monic polynomials with the numerator of degree $m$ and the denominator of degree $m+1$, Equations \eqref{Eq-3.33} and \eqref{Eq-3.3} together show that $\sum_{k=1}^{m+1} t_k = 1$. Indeed, by using \eqref{Eq-3.3}, we have
\begin{align*}
	t_k=\lim_{z\rightarrow z_k}\frac{(z-z_k)\phi(z)}{z\phi(z)-1}=\lim_{z\rightarrow z_k}\frac{\phi(z)+(z-z_k)\phi^{\prime}(z)}{\phi(z)+z\phi^{\prime}(z)}=\frac{\phi(z_k)}{\phi(z_k)+z_k\phi^{\prime}(z_k)}
\end{align*}
which yields that
\begin{align*}
	t_k = \frac{1}{1 + z_k \frac{\phi'(z_k)}{\phi(z_k)}}.
\end{align*}
By using [13, (3.4.7)], we have
\begin{align*}
	z_k \frac{\phi'(z_k)}{\phi(z_k)}>0
\end{align*}
which shows that $0 < t_k < 1$. Then, it follows that $\sum_{k=1}^{m+1} t_k = 1$ (see  \cite{Garcia-Mashreghi-Ross-2018}). In view of the last observation, by using \eqref{Eq-2.2} and \eqref{Eq-3.3}, we obtain
\begin{align}\label{Eq-3.4}
	\frac{f^{\prime\prime}(z)}{f^{\prime}(z)}=\alpha\frac{\phi(z)}{(1-z\phi(z))}=-\alpha\sum_{k=1}^{m+1}\frac{t_k}{z-z_k}=-\alpha\sum_{k=1}^{m+1}\frac{t_k(-\overline{z_k}))}{z-\overline{z_k}z},
\end{align}
which by integration yields that
\begin{align*}
	\log f^{\prime}(z)=-\alpha\sum_{k=1}^{m+1}t_k\log(1-\zeta_k z),\;\emph{i.e.,}\;f^{\prime}(z)=\prod_{k=1}^{m+1}\frac{1}{(1-\zeta_k z)^{\alpha t_k}},
\end{align*}
where $\zeta_k=\overline{z_k}$.
\subsection*{Sufficiency.} We assume that 
\begin{align*}
	f^{\prime}(z)=\prod_{k=1}^{m+1}\frac{1}{(1-\zeta_k z)^{\alpha t_k}},
\end{align*} 
where $\zeta_k \in \mathbb{T}$ are distinct points, $0 < t_k < 1$ and $\sum_{k=1}^{m+1} t_k = 1$. In view of \eqref{Eq-3.4}, we see that
\begin{align}\label{Eq-3.5}
		\alpha\frac{\phi(z)}{1-z\phi(z)}=\frac{f^{\prime\prime}(z)}{f^{\prime}(z)}=-\alpha\sum_{k=1}^{m+1}\frac{t_k}{z-\overline{\zeta_k}}\;\mbox{and}\;\phi(0)=\sum_{k=1}^{m+1}t_k\zeta_k.
\end{align}
Since $0 < t_k < 1$ and $\zeta_k \in \mathbb{T}$ are distinct points, we have $\phi(0) \in \mathbb{D}$. By  \cite[Theorem 3.5.2]{Garcia-Mashreghi-Ross-2018}, it suffices to prove that $\lim_{|z|\to 1^-} |\phi(z)|=1$. In order to prove this, we consider the first equation \eqref{Eq-3.5} and obtain that
\begin{align}\label{Eq-3.6}
	z\phi(z)=\frac{\sum_{k=1}^{m+1}\frac{zt_k}{z-\overline{\zeta_k}}}{\sum_{k=1}^{m+1}\frac{zt_k}{z-\overline{\zeta_k}}-1}
\end{align}
for $1\leq k\leq m+1$, we have 
\begin{align*}
	\lim_{z\rightarrow \zeta_k}z\phi(z)=1.
\end{align*}
Moreover, for $|a|=1$, $a \neq \zeta$ and $|\zeta|=1$, one has
\begin{align*}
		{\rm Re  }\left(\frac{\zeta}{\zeta -\overline{a}}\right) = \text{Re}\left(\frac{a\zeta}{a\zeta - 1}\right) = \frac{1}{2}.
\end{align*}
Thus, it is clear that
\begin{align*}
	{\rm Re  }\left(\sum_{k=1}^{m+1}\frac{t_k\zeta}{\zeta-\overline{\zeta_k}}\right)=\frac{1}{2}\sum_{k=1}^{m+1}t_k=\frac{1}{2}\;\mbox{for all}\; \zeta\in\mathbb{T}\setminus\{ {\overline{\zeta_1},\overline{\zeta_2},\ldots,\overline{\zeta_k}}\}.
\end{align*}
It follows from (2.8) that $\lim_{z\to\zeta} |z\phi(z)| = 1$ for all $\zeta \in \mathbb{T} \setminus \{\zeta_1, \zeta_2, \ldots, \zeta_{m+1}\}$. Hence, $\lim_{|z|\to 1^-} |\varphi(z)| = 1$ and the proof is completed.
\end{proof}
\section{\bf Pre-Schwarzian derivative and Schwarzian derivative for functions of the class $\mathcal{F}(\alpha)$}\label{Sec-4}
A series of results have been established by using the relationship between the univalence of a locally univalent analytic function and its Schwarzian derivative or pre-Schwarzian derivative. The origin of such an approach is connected with Nehari's investigations \cite{Nehari-BAMS-1949} using the Schwarzian derivative. Subsequently, this idea has been significantly developed by a number of researchers. We refer to the articles \cite[Section 8.5]{Duren-1983}, and \cite{Chuaqui-Duren-Osgood-AAS-2011,Pommerenke-1975} for further detail. For a locally univalent analytic function $f$ in $\mathbb{D}$, we define the pre-Schwarzian
derivative $P_f$ and the Schwarzian derivative $S_f$ by
\begin{align*}
	P_f(z) = \frac{f''(z)}{f'(z)}
\end{align*}
and
\begin{align*}
	S_f(z) := P_f'(z) - \frac{1}{2} P_f^2(z) = \frac{d}{dz} \left( \frac{f''(z)}{f'(z)} \right) - \frac{1}{2} \left( \frac{f''(z)}{f'(z)} \right)^2,
\end{align*}
respectively. Note that $P_f$ can be derived from the Jacobian $J_f = |f'|^2$ of $f$, namely,
\begin{align*}
	P_f(z) = \frac{\partial}{\partial z} (\log J_f).
\end{align*}

\noindent Their norms are defined by
\begin{align*}
	\|P_f\| = \sup_{z \in \mathbb{D}} (1 - |z|^2)|P_f(z)| \quad \text{and} \quad \|S_f\| = \sup_{z \in \mathbb{D}} (1 - |z|^2)^2|S_f(z)|,
\end{align*}
respectively. There are several well-known results which ensure that $f$ is univalent in $\mathbb{D}$ involving
these two quantities, and, possibly, has an extension to $K$-quasiconformal mapping $\phi$
of the extended complex plane $\widehat{\mathbb{C}} = \mathbb{C} \cup \{\infty\}$ onto itself, where $K = (1+k)/(1-k)$.
That is, there exists a quasiconformal homeomorphism $\phi$ on $\widehat{\mathbb{C}}$ such that $\phi(z) = f(z)$
for $z \in \mathbb{D}$ and
$$ \mu_{\infty} := \operatorname{ess\,sup}\{|\mu(z)| : z \in \mathbb{C}\} \leq k < 1, $$
where $\mu = \phi_{\bar{z}}/\phi_z$. For some historical and further discussion on these derivatives, we
refer to \cite{Agrawal-Sahoo-IJPAM-2021,Kim-Sugawa-RMJ-2002,Ponnusamy_sahoo-Sugawa-A-2014} and the references therein.
\begin{defi}
A quasicircle is a Jordan curve (a simple closed curve, meaning it doesn't intersect itself) in the complex plane that is the image of a circle (typically the unit circle) under a quasiconformal mapping of the plane onto itself.
\end{defi}
The interior of a quasicircle is called a quasidisk. So, a quasidisk is the image of a disk under a quasiconformal mapping. We remark that every bounded quasidisk is known to be a John disk, but not the converse. A crucial characterization is that a simply connected domain is a quasidisk if, and only if, it is both a John domain and a linearly connected domain (or equivalently, its boundary is a quasicircle). This highlights the importance of John domains in understanding quasiconformal mappings and the geometry of their images.\vspace{2mm}

  For pre-Schwarzian derivative, the following Becker univalence criterion \cite{Becker-JRAM-1972} is much deeper (cf. \cite{Becker-JRAM-1972,Becker-Prmmerenke-AAS-1984,Duren-Shapiro-Shields-DMJ-1966}).
\begin{theoA}\cite{Becker-JRAM-1972,Becker-Prmmerenke-AAS-1984}
	If $\|P_h\| \leq 1$, then $h$ is univalent in $\mathbb{D}$ and the constant 1 is best
	possible. Moreover, if $\|P_h\| \leq k < 1$, then $h$ has a continuous extension $\tilde{h}$ to $\overline{\mathbb{D}}$ and
	$h(\partial \mathbb{D})$ is a quasicircle.
\end{theoA}
Indeed, Becker showed that, if $\|P_h\| \leq k < 1$, then $h$ has a $K$-quasiconformal
extension to the whole complex plane $\mathbb{C}$, where $K = (1 + k)/(1 - k)$. In \cite{Obradovic-Ponnusamy-Wirths-SJM-2013}, it is
proved that $\|P_h\| \leq 2\alpha$ for $h \in \mathcal{G}(\alpha)$, where
\begin{align*}
	\mathcal{G}(\alpha)=\bigg\{h\in\mathcal{A} : {\rm Re}\left(\frac{zh''(z)}{\alpha h'(z)}\right)<\frac{1}{2}\; \mbox{for each}\; z\in\mathbb{D}\bigg\}
\end{align*}
and $0<\alpha\leq 1$. In \cite{Li-Pon-BMMSS-2025}, Li and Ponnusamy  show that the family $\mathcal{G}(\alpha)$ has several elegant properties, including those involving Blaschke products, Schwarzian derivative and univalent
harmonic mappings.\vspace{2mm} 

Inspired by the article \cite{Li-Pon-BMMSS-2025}, for the class $f\in\mathcal{F}(\alpha)$, it is natural to raise the following.
\begin{ques}\label{Q-4.1}
	Is $f(\mathbb{D})$ a quasidisk when $f\in\mathcal{F}(\alpha)$?
\end{ques}

Following Becker's result, Question \ref{Q-4.1} is answered affirmatively, leading to the following:
\begin{thm}\label{Th-4.1}
	Suppose that $f\in\mathcal{F}(\alpha)$ for some $\alpha\in(0,3]$. Then $f(\mathbb{D})$ is a quasidisk.
\end{thm}
\begin{rem}
	Note also that for the choice of $\alpha\in (0,1)$, $f \in \mathcal{F}(\alpha)$ implies that $\|P_f\| \leq 2\alpha<2$ and hence,
	$f(\mathbb{D})$ is a John disk.
\end{rem}
\begin{proof}
	Let $f \in \mathcal{F}(\alpha)$. According to Theorem \ref{Th-3.1}, if $f \in \mathcal{F}(\alpha)$, then $f'(z)$ is subordinate to the function $H_{\alpha}(z) = (1-z)^{-\alpha}$. This implies that the image of the unit disk under the derivative, $f'(\mathbb{D})$, is contained within the domain $\mathcal{D} = H_{\alpha}(\mathbb{D})$, where $H_\alpha(z) = (1-z)^{-\alpha}$. The domain $\mathcal{D} = H_{\alpha}(\mathbb{D})$ has specific geometric characteristics for $\alpha \in (0, 3]$:
	\begin{enumerate}
		\item[(a)] It is a bounded domain.
		\item[(b)] It is contained within the right half-plane $\mathbb{H} = \{w : Re(w) > 0\}$.
		\item[(c)] It is a `quasidomain' in the sense that it is far enough from the boundary of the half-plane to allow for certain extension properties.
	\end{enumerate}
	Let $K$ be a closed disk with center in $\mathbb{H} \setminus \overline{\mathcal{D}}$. Then $K$ is a compact subset and for each
	$r > 0$, one has
	$$ rK \cap (\mathbb{H} \setminus \mathcal{D}) \neq \emptyset. $$
	Moreover, as $h'(z) \prec H_c(z)$, we have $h'(\mathbb{D}) \subset \mathcal{D}$. This fact and the result in \cite{Chuaqui-Gervirtz-CVTA-2003} imply
	that $h(\mathbb{D})$ is a quasidisk. More specifically, 
	\begin{enumerate}
		\item[(a)] A domain $f(\mathbb{D})$ is a quasidisk if and only if it is both a John domain and a linearly connected domain.\vspace{1.2mm}
		
		\item[(b)] A classic result (referenced as \cite{Chuaqui-Gervirtz-CVTA-2003}) states that if $f'(z)$ is subordinate to a function whose image is a bounded domain in the right half-plane, then the integral $f(z) = \int f'(\zeta) d\zeta$ maps $\mathbb{D}$ onto a quasidisk.\vspace{1.2mm}
		
		\item[(c)] Because $f'(\mathbb{D}) \subset \mathcal{D}$ and $\mathcal{D}$ is a bounded domain in $\mathbb{H}$, the function $f$ satisfies the criteria for its image to be a quasidisk.
	\end{enumerate}
	This completes the proof.
\end{proof}

Next we recall the following well-known results from \cite{Nehari-BAMS-1949} (see also \cite{Nehari-PAMS-1954,Nehari-IJM-1979} which
deals with general situation). As shown by Hille \cite{Hille-BAMS-1949}, the constant $``2"$ below is best
possible.
\begin{theoB}
	If $f \in \mathcal{S}$, then we have the sharp inequality $\|S_h\| \leq 6$ and the number 6
	is best possible. Conversely, if $\|S_h\| \leq 2$, then $h$ is univalent in $\mathbb{D}$ and the number 2
	is best possible. Moreover, if $\|S_h\| \leq k < 2$, then $h$ has a continuous extension to the
	whole complex plane.
\end{theoB}
\begin{lem}\label{Lem-4.1}
	Let $f\in\mathcal{F}(\alpha)$ for some $\alpha\in(0,2)$ such that $f^{\prime}(z)=(1-z\zeta)^{-\alpha}$ for all $z\in\mathbb{D}$ and $\zeta\in\mathbb{T}$. Then $\|S_f\|=2\alpha(2-\alpha)$.
\end{lem}
\begin{proof}
	Proof follows from a direct computation. Indeed, for the function $f'(z) = (1 - \zeta z)^{-\alpha}$, we have
	\begin{align*}
		S_f(z) =  \frac{\alpha\zeta^2}{(1 - \zeta z)^2} - \frac{1}{2} \frac{\alpha^2\zeta^2}{(1 - \zeta z)^2} = \frac{\alpha(2-\alpha)}{2} \frac{\zeta^2}{(1 - \zeta z)^2}
	\end{align*}
	showing that $\|S_f\|=2\alpha(2-\alpha)$
\end{proof}
In the next result, we determine the sharp bound for the Schwarzian derivative of $f\in\mathcal{F}(\alpha).$
\begin{thm}
	Suppose that $f\in\mathcal{F}(\alpha)$ for some $\alpha\in (0,2)$. Then $\|S_f\| \leq 2\alpha(2-\alpha)$, where the equality is attained by $f(z)$ which is obtained from $f^{\prime}(z) = (1 - \zeta z)^{-\alpha}$ for some $\zeta \in \mathbb{T}$.
\end{thm}
\begin{proof}
	Let $f\in\mathcal{F}(\alpha)$. By using \eqref{Eq-2.2}, we see that 
	\begin{align*}
		\frac{f^{\prime\prime}(z)}{f^{\prime}(z)}=\frac{\alpha\phi(z)}{1-z\phi(z)}
	\end{align*}
	for some $\phi\in\mathbb{B}$. Therefore, by Schwarz-Pick inequality, we have 
	\begin{align}\label{Eq-4.1}
		(1-|z|^2)|\phi^{\prime}(z)|\leq1-|\phi(z)|^2\;\mbox{for}\;z\in\mathbb{D}.
	\end{align}
	Now, a calculation yields
	\begin{align*}
		S_f(z)=\alpha\left(\frac{\phi^{\prime}(z)+(1-\alpha/2)\phi^2(z)}{\left(1-z\phi(z)\right)^2}\right)
	\end{align*}
	so that
	\begin{align*}
		\frac{1}{\alpha}|S_f|(1-|z|^2)^2\leq\frac{|\phi^{\prime}(z)(1-|z|^2)^2}{|1-z\phi(z)|^2}+\frac{(1-\alpha/2)|\phi(z)|^2(1-|z|^2)^2}{|1-z\phi(z)|^2}
	\end{align*}
	which, by using the Schwarz-Pick inequality \eqref{Eq-4.1}, reduces to
	\begin{align}\label{Eq-4.2}
		\frac{1}{\alpha}|S_f|(1-|z|^2)^2\leq\frac{(1-|\phi(z)|)(1-|z|^2)}{|1-z\phi(z)|^2}+\frac{(1-\alpha/2)|\phi(z)|^2(1-|z|^2)^2}{|1-z\phi(z)|^2}.
	\end{align}
	If there exists a $z_0\in\mathbb{D}$, such that $|\phi(z_0)|=1$, then $\phi(z)\equiv\xi$ for some $\xi\in\mathbb{T}$ and $f^{\prime}(z)=(1 - \zeta z)^{-\alpha}$. By Lemma \ref{Lem-4.1}, we obtain that 
	\begin{align*}
		\|S_f\|=2\alpha(2-\alpha).
	\end{align*}
	If $|\phi(z)| < 1$ for all $z \in \mathbb{D}$, then, by the inequality \eqref{Eq-4.2}, we only need to prove that
	\begin{align*}
		\frac{(1-|\phi(z)|)(1-|z|^2)}{|1-z\phi(z)|^2}+\frac{(1-\alpha/2)|\phi(z)|^2(1-|z|^2)^2}{|1-z\phi(z)|^2}\leq2(2-\alpha),
	\end{align*}
	which equivalent to
	\begin{align*}
		&2(1-|z|^2)-2(1 - |z|^2)|\phi(z)|^2 + (2-\alpha)(1 - 2|z|^2+|z|^4)|\phi(z)|^2\\&\leq4(2-\alpha) \left[ 1 - 2\text{Re} (z\phi(z))+|z|^2|\phi(z)|^2 \right]
	\end{align*}
	Simplifying the last inequality, one obtains an equivalent inequality $M(\alpha)\leq0$, where
	\begin{align*}
		M(\alpha)&=\left((2-\alpha)|z|^4 - (10-6\alpha)|z|^2-\alpha\right) |\phi(z)|^2 + 8(2-\alpha)\text{Re} (z\phi(z))\\&\quad-2\left(|z|^2-2\alpha + 3\right).
	\end{align*}
	\textbf{Claim:} $M(\alpha)$ is decreasing with respect to $\alpha$ for $\alpha\in[0, 2]$.
	By the assumption, we have the following estimate
	\begin{align*}
		M^{\prime}(\alpha)&=\left(-|z|^4+6|z|^2-1\right)|\phi(z)|^2-8\text{Re}(z\phi(z))+4\\&=-\left(1-|z|^2\right)^2|\phi(z)|^2+4|z\phi(z)-1|^2\\&=\left(2|z\phi(z)-1|+(1-|z|^2)|\phi(z)|\right)\left(2|z\phi(z)-1|-(1-|z|^2)|\phi(z)|\right)
	\end{align*}
	Now, we let 
	\begin{align*}
		N(\alpha)=\left(2|z\phi(z)-1|+(1-|z|^2)|\phi(z)|\right).
	\end{align*}
	It is easy to see that
	\begin{align*}
		N(\alpha)\leq-(1-|z|^2)|\phi(z)|-2(1-|\phi(z)|)<0,\;z\in\mathbb{D}.
	\end{align*}
	The above estimate together with the expression of $M^{\prime}(\alpha)$ yields
	\begin{align*}
		M^{\prime}(\alpha)<0,\;z\in\mathbb{D}
	\end{align*}
	and the claim is proved.\vspace{2mm}
	
	Consequently, we have $M(\alpha)\leq M(0)$ for $\alpha\in (0,2]$. To complete the proof, it suffices to show that $M(0)\leq 0$ for all $z \in \mathbb{D}$. We see that
	\begin{align*}
		M(0)=\left(2|z|^4 - 10|z|^2\right) |\phi(z)|^2 + 16\text{Re} (z\phi(z)) - 2\left( |z|^2 + 3 \right)
	\end{align*}
	so that
	\begin{align*}
		M(0)/2&=\left(|z|^4 - |z|^2\right) |\phi(z)|^2-4|\phi(z)|^2 +8\text{Re} (z\phi(z)) - \left( |z|^2 + 3 \right)\\&=(1-|z|)(1-|z\phi(z)|)-4|z\phi(z)-1|^2\\&<4(1-|z|)(1-|z\phi(z)|)-4|z\phi(z)-1|^2\\&<4|1-z\phi(z)|\left((1-|z|)-|1-z\phi(z)|\right)<0.
	\end{align*}
	This completes the proof.
\end{proof}
\section{\bf Results on univalent harmonic mappings with analytic part from $\mathcal{F}(\alpha)$}\label{Sec-5}
A complex-valued function $ f $ in $ \mathbb{D} $ is said to be harmonic if it satisfies the Laplace equation $ \Delta f=4f_{z\bar{z}}=0 $ in $ \mathbb{D} $. Every harmonic mapping $ f $ in $ \mathbb{D} $ has the unique canonical form $ f=h+\bar{g} $, where $ h $ and $ g $ are analytic in $ \mathbb{D} $ with $ g(0)=0 $. Every analytic function is a harmonic mapping. Let $ \mathcal{H} $ be the class of all complex-valued harmonic mappings $ f=h+\bar{g} $ defined on $ \mathbb{D} $, where $ h $ and $ g $ are analytic in $ \mathbb{D} $ with the normalization $ h(0)=h^{\prime}(0)-1=0 $ and $ g(0)=0. $ Here $ h $ is called analytic part and $ g $ is called co-analytic part of $ f $. Moreover, the mapping $ f=h+\bar{g} $ has the following representation
\begin{align*}
	f(z)=h(z)+\overline{g(z)}=\sum_{n=2}^{\infty}a_nz^n+\overline{\sum_{n=1}^{\infty}b_nz^n}.
\end{align*}
Let $ \mathcal{S}_{\mathcal{H}} $ be the subclass of $ \mathcal{H} $ consisting of univalent mappings and we observe that $ \mathcal{S}_{\mathcal{H}} $ reduces to the class $ \mathcal{S} $ of normalized univalent functions, if the co-analytic part $ g\equiv 0 $. A function $ f\in\mathcal{H} $ is said to be close-to-convex in $ \mathbb{D} $ if it is univalent and image $ f(\mathbb{D}) $ is a close-to-convex domain, \textit{i.e.,} the complement of $ f(\mathbb{D}) $ can be written as the union of non-intersecting half-lines. We denote $ \mathcal{K}_{\mathcal{H}} $ the close-to-convex subclass of $ \mathcal{S}_{\mathcal{H}} $. Lewy \cite{Lewy-BAMS-1936} proved that $ f=h+\bar{g} $ is locally univalent in $ \mathbb{D} $ if, and only if, the Jacobian $ J_f:=|h^{\prime}|^2-|g^{\prime}|^2\neq 0 $ on $ \mathbb{D} $. Moreover, the harmonic mapping $ f=h+\bar{g} $ is sense-preserving if $ J_f>0 $ \textit{i.e.,} $ |h^{\prime}|>|g^{\prime}| $, and at this point, its dilatation $ \omega_f:=g^{\prime}/h^{\prime} $ has the property $ |\omega_f|<1 $ on $ \mathbb{D} $.\vspace{2mm}

Harmonic mappings play the natural role in parameterizing minimal surfaces in the context of differential geometry. Planer harmonic mappings have application not only in the differential geometry but also in various fields of engineering, physics, operations research and other intriguing aspects of applied mathematics. The theory of harmonic mappings has been used to study and solve fluid flow problems \cite{aleman-2012}. The theory of univalent harmonic mappings having prominent geometric properties like starlikeness, convexity and close-to-convexity appear naturally while dealing with planner fluid dynamical problems. For instance, the fluid flow problem on a convex domain satisfying an interesting geometric property has been extensively studied by Aleman and Constantin \cite{aleman-2012}. With the help of geometric properties of harmonic mappings, Constantin and Martin \cite{constantin-2017} have obtained a complete solution of classifying all two-dimensional fluid flows.\vspace{2mm}

If a locally univalent and sense-preserving harmonic mapping $f = h +\overline{g}$ on $\mathbb{D}$ satisfies the condition $|\omega_f(z)| \leq k < 1$ for $\mathbb{D}$, then $f$ is called a $K$-quasiregular harmonic mapping in $\mathbb{D}$, where $K = \frac{1+k}{1-k} \geq 1$. We refer to [9] for several properties of univalent harmonic mappings together with its various subfamilies. In particular, here is a sufficient condition for close-to-convexity of harmonic mappings due to Clunie and Sheil-Small \cite{Clunie-Sheil-AASFS-2007} (see also \cite{Duren-Har-1983}).\vspace{2mm}

Since harmonic mappings are natural extension of analytic maps on unit disk $\mathbb{D}$, in this section, we introduce a new subclass $\mathcal{F}_{\mathcal{H}}(c),\;c>0$ of harmonic mappings which is defined by 
\begin{align*}
	\mathcal{F}_{\mathcal{H}}(\alpha):=\{f=h+\bar{g} \in\mathcal{H} : h\in \mathcal{F}(\alpha)\; \mbox{and}\;g^{\prime}=\omega h^{\prime}\}.
\end{align*}
It is easy to see that $\mathcal{F}(\alpha)\subseteq \mathcal{F}_{\mathcal{H}}(\alpha)$ in $ \mathbb{D}$ in the sense that harmonic mappings of the form $f=h+\bar{g}$ is a generalization of analytic functions $h$. The Schwarzian norm estimate is understood from the article \cite{Ahamed-Hossain-2024,Wang-Li-Fan-MM-2024} but the Schwarzian norm or pre-Schwarzian norm estimate is not explored yet for a class of harmonic mappings related with the class $\mathcal{F}(\alpha)$. \vspace{2mm}

In complex function theory, the `John domain' is a significant concept that describes a certain class of domains with specific geometric properties. It was introduced by Fritz John in his work on elasticity and later found profound applications in various areas of mathematics, particularly in complex analysis, geometric function theory, and the theory of quasiconformal mappings.\vspace{2mm}

A common definition for a bounded simply connected domain $\Omega$ is:

\begin{defi}
	$\Omega$ is a John domain if there exist constants $a > 0$ and $x_0 \in \Omega$ such that for every $x \in \Omega$, there is a rectifiable curve $\gamma$ in $\Omega$ connecting $x$ to $x_0$ such that for every point $y$ on $\gamma$, the distance from $y$ to the boundary of $\Omega$, denoted by $\text{dist}(y, \partial\Omega)$, satisfies:
	$$\text{dist}(y, \partial\Omega) \ge a \cdot \text{length}(\gamma(x,y))$$
	where $\gamma(x,y)$ is the subarc of $\gamma$ between $x$ and $y$.
\end{defi} John domains are closely related to other important classes of domains in complex analysis such as Linearly connected domains, quasidisks.\vspace{2mm}
\begin{figure}
	\begin{center}
		\includegraphics[width=0.99\linewidth]{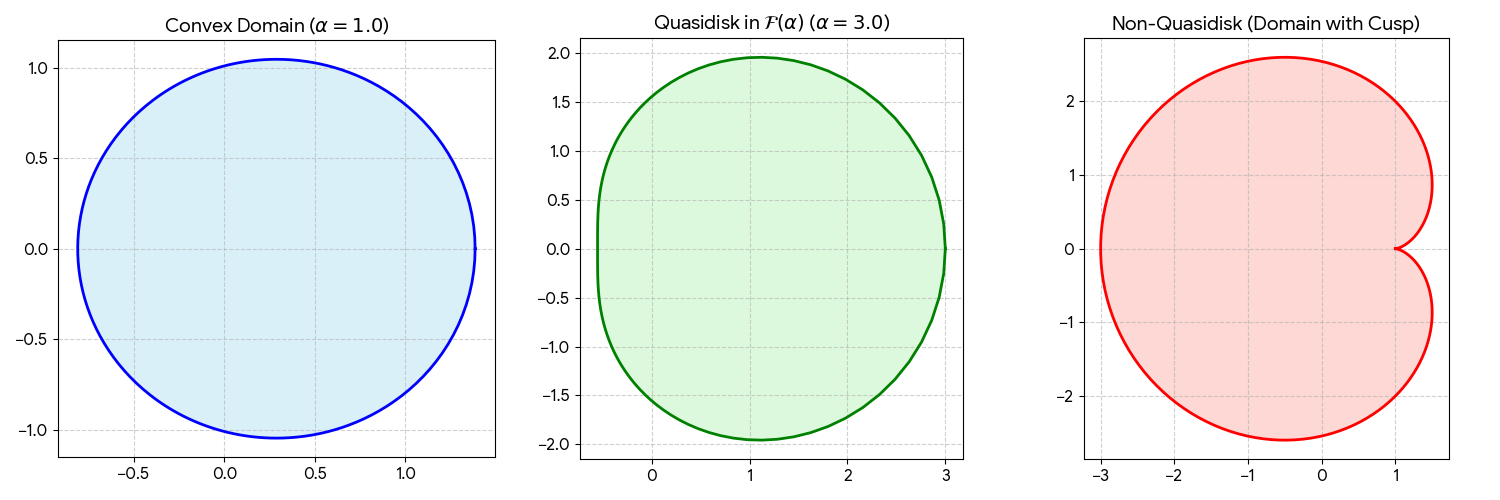}
	\end{center}
	\caption{In the left panel ($\alpha=1$), we observe a smooth, convex boundary, representing the most fundamental example of a quasidisk. The middle panel ($\alpha=3$) illustrates a domain that, while no longer convex, remains a quasidisk; its boundary is a Jordan curve free of cusps, satisfying the `bounded turning' property we established. Conversely, the right panel depicts a non-quasidisk domain with a sharp inward cusp, where the boundary violates the Ahlfors condition. Our results prove that functions in $\mathcal{F}(\alpha)$ effectively avoid such pathological behavior.}
\end{figure}
\begin{figure}
	\begin{center}
		\includegraphics[width=0.99\linewidth]{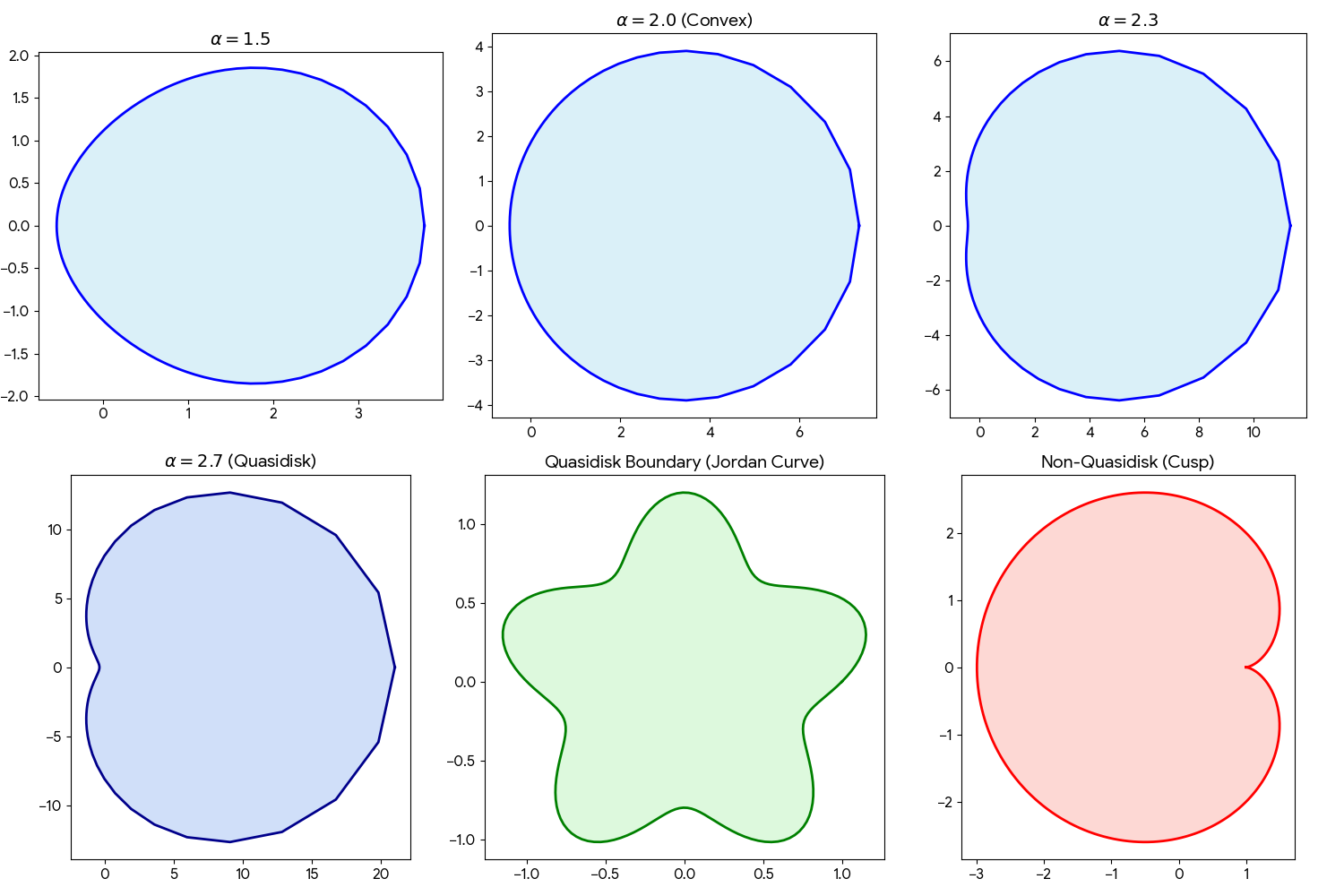}
	\end{center}
	\caption{This figure illustrates the geometric evolution of image domains $f(\mathbb{D})$ in the family $\mathcal{F}(\alpha)$. Panels (i) ($\alpha = 1.5$) and (ii) ($\alpha = 2.0$) depict highly regular geometries, with the latter representing the classical convex case. As $\alpha$ increases to (iii) $2.3$ and (iv) $2.7$, the domains lose convexity but remain univalent quasidisks, as established in Theorem \ref{Th-4.1}. While a general quasidisk (v) may feature a non-smooth Jordan curve satisfying the Ahlfors condition, it remains distinct from the non-quasidisk domain (6), which possesses a sharp inward cusp that violates the bounded turning property. Our results prove that for $\alpha \in (0, 3]$, functions in $\mathcal{F}(\alpha)$ maintain boundary regularity and avoid such pathological singularities.}
\end{figure}
\begin{figure}
	\begin{center}
		\includegraphics[width=0.99\linewidth]{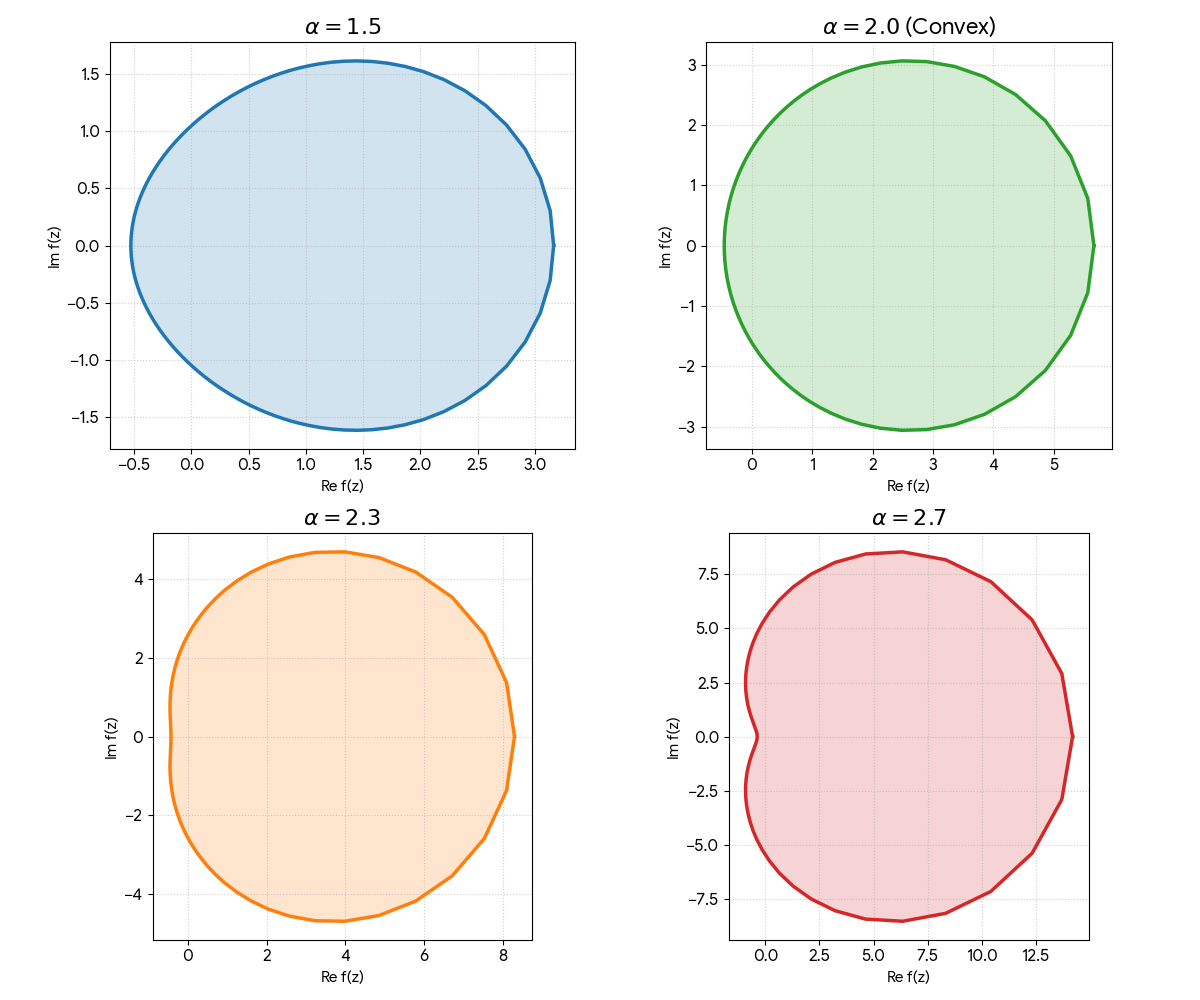}
	\end{center}
	\caption{\textbf{Top Row:} $\alpha = 1.5$: Represents a smooth, highly regular domain typical of the close-to-convex subfamily.$\alpha = 2.0$: The critical convex case. At this point, the domain is strictly convex, and its boundary is a quintessential example of a `tame' quasicircle. \textbf{Bottom Row:} $\alpha = 2.3$: The domain begins to elongate significantly. While convexity is lost, the boundary remains a smooth Jordan curve, preserving univalence.$\alpha = 2.7$: Near the upper limit ($\alpha=3$), the domain shows extreme stretching toward the singularity at $z=1$. Despite this `sharpness,' our results in Theorem 4.1 prove that the image is still a quasidisk, as the boundary satisfies the Ahlfors condition and avoids any zero-angle cusps.}
\end{figure}

Recall that a domain $\mathcal{D}$ is linearly connected if there exists a positive constant $M < \infty$ such that any two points $z,w \in \mathcal{D}$ are joined by a path $\gamma \subset \mathcal{D}$ of length (cf. [7])
$$ \ell(\gamma) \leq M|z - w|. $$
or equivalently $\text{diam}(\gamma) \leq M|z - w|$. We point out that a bounded linearly connected domain is a Jordan domain, and for piecewise smoothly bounded domains, linear connectivity is equivalent to the boundary's having no inward-pointing cusps. Also, if this inequality holds with $M = 1$, then the linearly connected domain is convex.
\begin{theo}
	Suppose that $f=h+\bar{g}\in\mathcal{F}_{\mathcal{H}}(\alpha)$ for some $\alpha\in (0, 3]$. Then there exists $\delta > 0$ such that every harmonic mapping $f = h +\overline{g}$ with dilatation $|\omega(z)| < \delta$, is univalent in $\mathbb{D}$, where the constant $c$ depends only on the domain $h(\mathbb{D})$.
\end{theo}
\begin{proof}
	Since $f(\mathbb{D})$ is a quasidisk by Theorem \ref{Eq-4.1}, $\mathcal{D} = f(\mathbb{D})$ is a linearly connected domain. The conclusion of the theorem follows from \cite[Theorem 1]{Chuaqui-Harnandez-JMAA-2007}, hence, we omit the details here.
\end{proof}
\begin{theo}
	Let $h \in \mathcal{F}(\alpha)$ for some $\alpha\in (0, 1/2)$, and $f = h +\overline{g}$ be a sense-preserving harmonic mapping with dilatation $\omega = g'/h'$. If
	\begin{align*}
		|\omega(z)|\leq1-\alpha|z|(1 + |z|)\quad\text{for all }z\in \mathbb{D},
	\end{align*}
	then $f$ is univalent in $\mathbb{D}$.
\end{theo}
\begin{proof}
	Since $h\in\mathcal{F}(\alpha)$, we have $h^{\prime}(z)\prec(1-z)^{-\alpha}$ by Theorem \ref{Th-3.1} and hence $h^{\prime}(z)\neq0$ for $z\in\mathbb{D}$. Recall equation \eqref{Eq-2.2}
	\begin{align*}
		\frac{h^{\prime\prime}(z)}{h^{\prime}(z)}=\frac{\alpha\phi(z)}{1-z\phi(z)},\;\mbox{for all}\;z\in\mathbb{D},
	\end{align*}
	where $\phi\in\mathbb{B}$. By computation, we obtain that 
	\begin{align*}
		(1-|z|^2)\bigg|\frac{zh^{\prime\prime}(z)}{h^{\prime}(z)}\bigg|\leq \alpha|z|(1+|z|)\;\mbox{for all}\;z\in\mathbb{D},
	\end{align*}
	which together with the result in \cite{Avkhadiev-Nasibullin-Shafigullin-RM-2016} implies that $f$ is univalent in $\mathbb{D}$.
\end{proof}
\section{\bf{Pre-Schwarzian derivatives for harmonic mapping with fixed analytic parts belong to $\mathcal{F}(\alpha)$}}\label{Sec-6}

The pre-Schwarzian derivative (denoted by \( P_f \)) of a locally univalent and sense-preserving harmonic mapping \( f = h + \bar{g} \), with dilatation \( \omega = {g'}/{h'} \), where \( \omega = q^2 \) for some analytic function \( q \), was introduced by Kanas and Klimek-Smęt~\cite{Kanas-Klimek-Smet-BKMS-2014}. It is defined as
\begin{align}\label{Eq-6.111}
	P_f = 2 \frac{\partial (\log \lambda)}{\partial z} = \frac{h''}{h'} + \frac{2q q'}{1 + |q|^2},
\end{align}
where \( \lambda = |h'| + |g'| \).\vspace{2mm}

However, the pre-Schwarzian derivative defined in \eqref{Eq-6.111} has a significant limitation arising from restrictions on the dilatation. Consequently, in many cases, it cannot be defined globally for a univalent harmonic mapping.\vspace{2mm}

To address this problem, Hernández and Martín~\cite{Hernández-Martín-JGA-2015} proposed a new definition of the Schwarzian derivative for harmonic mappings in 2015. For a locally univalent harmonic mapping \( f = h + \bar{g} \), the Schwarzian derivative is defined as
\begin{align*}
	S_f = (\log J_f)_{zz} - \frac{1}{2} \left( (\log J_f)_z \right)^2,
\end{align*}
which can be expressed in terms of the analytic part \( h \) and the dilatation \( \omega \) as
\begin{align*}
	S_f = S_h + \frac{\omega}{1 - |\omega|^2} \left( \frac{h''}{h'} \omega' - \omega'' \right)
	- \frac{3}{2} \left( \frac{\omega' \omega}{1 - |\omega|^2} \right)^2,
\end{align*}
where \( J_f \) is the Jacobian of \( f \), \( S_h \) is the classical Schwarzian derivative of the analytic function \( h \), and \( \omega = {g'}/{h'} \) is the dilatation of \( f \).\vspace{2mm}

The corresponding pre-Schwarzian derivative in this framework is given by
\begin{align*}
	P_f = (\log J_f)_z = \frac{h''}{h'} - \frac{\omega \omega'}{1 - |\omega|^2}.
\end{align*}

In recent developments, Wang (2024) (see, \cite{Wang-Li-Fan-MM-2024}), along with Ahamed and Hossain (see, \cite{Ahamed-Hossain-2024}), derived estimates for the pre-Schwarzian derivative of harmonic mappings with a fixed analytic part. Motivated by the work of Ahamed and Hossain, as well as the foundational contributions of Hernández and Mart\'in to the theory of harmonic Schwarzian derivatives (see, \cite{Hernández-Martín-JGA-2015}), we investigate a related class of harmonic mappings $\mathcal{F}_{\mathcal{H}}(\alpha)$.\vspace{2mm}

Our main aim is to derive estimates for the modulus of the pre-Schwarzian norm for the functions in $\mathcal{F}_{\mathcal{H}}(c)$. In fact, we obtain the following result which will provide a sharp estimate of the pre-Schwarzian norm for functions in the class $\mathcal{F}_{\mathcal{H}}(\alpha)$.
\begin{thm}\label{Th-6.1}
	If $f=h+\bar{g}\in\mathcal{F}_{\mathcal{H}}(\alpha)$ with $g^{\prime}=\omega h^{\prime}$, then $\|P_{f}\|\leq 2\alpha+1$. The estimate $2\alpha+1$ is best possible.
\end{thm}
\begin{proof}[\bf Proof of Theorem \ref{Th-6.1}]
	Let $f=h+\bar{g}\in\mathcal{F}_{\mathcal{H}}(\alpha)$. Then $h$ holds the subordination relation 
	\begin{align*}
		1+\frac{zh^{\prime\prime}(z)}{h^{\prime}(z)}\prec1+\frac{\alpha z}{1-z},
	\end{align*}
	which gives us
	\begin{align}\label{Eq-6.1}
		\bigg|\frac{h^{\prime\prime}(z)}{h^{\prime}(z)}\bigg|\leq\frac{\alpha}{1-|z|}.
	\end{align}
	By the Schwarz-Pick lemmma and \eqref{Eq-6.1}, we have 
	\begin{align*}
		\|P_{f}\|&=\sup_{z\in\mathbb{D}}\left(1-|z|^2\right)|P_{f}|\\&=\sup_{z\in\mathbb{D}}\left(1-|z|^2\right)\bigg|\frac{h^{\prime\prime}(z)}{h^{\prime}(z)}-\frac{\bar{\omega}(z)\omega^{\prime}(z)}{1-|\omega(z)|^2}\bigg|\\&\leq\sup_{z\in\mathbb{D}}\left(1-|z|^2\right)\left(\bigg|\frac{h^{\prime\prime}(z)}{h^{\prime}(z)}\bigg|+\bigg|\frac{\bar{\omega}(z)\omega^{\prime}(z)}{1-|\omega(z)|^2}\bigg|\right)\\&\leq\sup_{z\in\mathbb{D}}\left(1-|z|^2\right)\left(\frac{\alpha}{1-|z|}+\frac{|\bar{\omega}(z)|}{1-|z|^2}\right)\\&\leq\sup_{z\in\mathbb{D}}\left(\alpha(1+|z|)+|\omega(z)|\right)\\&=2\alpha+1.
	\end{align*}
	We now show that the estimate is best possible. For
	\[ \begin{cases}
		\sqrt{1-\alpha}\leq t<1,\;\;\;\mbox{when}\;\;\alpha\in[0,1)\vspace{2mm}\\\dfrac{1-2\alpha}{3-2\alpha}\leq t<1,\;\;\;\;\mbox{when}\;\;\alpha\in[1,\infty),
	\end{cases}
	\]
	we consider the function $f_t=h_t+\bar{g_t}\in\mathcal{H}$ with the second complex dilatation $\omega_t(z)=(z-t)/(1-tz)$ and $h(z)$ be such that
	\begin{align*}
		1+\frac{zh_t^{\prime\prime}(z)}{h_t^{\prime}(z)}=1+\frac{\alpha z}{1-z}.
	\end{align*}
	Then clearly $f_t\in\mathcal{F}_{\mathcal{H}}(\alpha)$ for all t in $[0,\infty)$. A simple computations yields that
	\begin{align*}
		\frac{\bar{\omega_t}(z)\omega_t(z)}{1-|\omega_t(z)|^2}=\frac{\bar{z}-t}{(1-tz)(1-|z|^2)}.
	\end{align*}
	Consequently, we have
	\begin{align*}
		\|P_{f_t}\|&=\sup_{z\in\mathbb{D}}(1-|z|^2)\bigg|\frac{h^{\prime\prime}_t(z)}{h^{\prime}_t(z)}-\frac{\bar{\omega}(z)\omega^{\prime}(z)}{1-|\omega(z)|^2}\bigg|\\&=\sup_{z\in\mathbb{D}}(1-|z|^2)\bigg|\frac{\alpha}{1-|z|}-\frac{\bar{z}-t}{(1-tz)(1-|z|^2)}\bigg|.
	\end{align*}
	We set
	\begin{align*}
		M_t=\sup_{z\in\mathbb{D}}(1-|z|^2)\bigg|\frac{\alpha}{1-|z|}-\frac{\bar{z}-t}{(1-tz)(1-|z|^2)}\bigg|
	\end{align*}
	\begin{align}\label{Eq-6.2}
		=\sup_{z\in[0,1)}\bigg|\alpha(1+r)-\frac{r-t}{1-tr}\bigg|=\sup_{z\in[0,1)}\;\phi(r),
	\end{align}
	where 
	\begin{align*}
		\phi(r)=\alpha(1+r)-\frac{r-t}{1-tr}.
	\end{align*}
	A simple computation shows that
	\begin{align*}
		\phi^{\prime}(r)=\alpha-\frac{1-t^2}{(1-rt)^2}\;\;\mbox{and}\;\;\phi^{\prime\prime}(r)=-\frac{2t(1-t^2)(1-rt)}{(1-rt)^4}<0\;\;\mbox{for all r}\;\in[0,1).
	\end{align*}
	Now, $\phi^{\prime}(r)=0$ gives
	\begin{align*}
		r=r_0:=\frac{1}{t}-\frac{1}{t}\sqrt{\frac{1-t^2}{\alpha}}.
	\end{align*}
	Thus the maximum value of $\phi$ is attained at $r_0$. Hence, we have
	\begin{align}\label{Eq-6.3}
		M_t=\phi(r_0)=\frac{1}{t}\left(1+2\alpha-2\alpha\sqrt{\frac{1-t^2}{\alpha}}\right).
	\end{align}
	In view of \eqref{Eq-6.2} and \eqref{Eq-6.3}, it is clear that $M_t\leq\|P_{f_t}\|\leq2\alpha+1$. We note that $M_t$ is increasing function for t and $M_t\rightarrow1+2\alpha$ as $t\rightarrow1$. This shows that estimate is the best possible.
\end{proof}
\section{\bf Estimate of Bloch constant for the class $\mathcal{F}_{\mathcal{H}}(\alpha)$}
Let $\mathcal{H}ar(\mathbb{D})$ denote the family of continuous complex-valued functions which are harmonic in the open unit disk $\mathbb{D} = \{z \in \mathbb{C} : |z| < 1\}$, and let $\mathcal{H}ol(\mathbb{D})$ denote the class of holomorphic functions $f$ in $\mathbb{D}$ withthe normalization $f(0) = f'(0) - 1 = 0$ in $\mathbb{D}$. We note that $\mathcal{H}ol(\mathbb{D}) \subset \mathcal{H}ar(\mathbb{D})$. Further, let $\mathcal{S} := \mathcal{S}_{\mathcal{H}ol}$ be the subclass of $\mathcal{H}ol(\mathbb{D})$ which are additionally univalent in $\mathbb{D}$. Clunie and Sheil-Small in \cite{Clunie- Sheil-Small-AAS-1984} developed the fundamental theory of functions $f \in \mathcal{H}ar(\mathbb{D})$ with the normalization $f(0) = h'(0) - 1 = 0$ in $\mathbb{D}$. Following Clunie and Sheil-Small's notation, we next denote by $\mathcal{S}_{\mathcal{H}ar}$, the subclass of $\mathcal{H}ar(\mathbb{D})$ consisting of univalent and sense-preserving harmonic mappings $f = h +\overline{g}$ in $\mathbb{D}$, where $h$ and $g$ are normalized such that
\begin{align}\label{Eq-7.1}
	h(z) = z + \sum_{n=2}^{\infty} a_n z^n\;\mbox{and}\;g(z) = \sum_{n=1}^{\infty} b_n z^n.
\end{align}
Here, $h$ and $g$ are called the analytic part and the co-analytic part of $f$, respectively.

The analytic parts of harmonic mappings play a vital role in shaping their geometric properties. For instance, if $f=h+\overline{g}$ is a sense-preserving harmonic mapping and $h$ is convex univalent, then $f \in \mathcal{S}_{\mathcal{H}ar}$ and maps $\mathbb{D}$ onto a close-to-convex domain \cite{Clunie- Sheil-Small-AAS-1984}. In \cite{Kanas-Klimek-Smet-BKMS-2014,Kanas-Klimek-Smet-BM-2016}, a class of functions $f=h+\overline{g}\in \mathcal{S}_{\mathcal{H}ar}$ has been studied, where $h$ and $g$ are given by (1.2), such that $b_1=\alpha/3\in (0, 1)$, $h$ is convex in $\mathbb{D}$ (or $h$ is a function with bounded boundary rotation) and the dilatation $\omega$ is of the form
\begin{align*}
	w(z) =(z+\alpha/3)/(1 + z\alpha/3)
\end{align*}

For $\alpha$ ($0 <\alpha/3<1$), let $\mathcal{F}_{\mathcal{H}}(\alpha)$ denote the set of all harmonic mappings $f=h+\overline{g} \in \text{Har}(\mathbb{D})$, with $g'(0) = b_1 =\alpha/3$ and
\begin{align}\label{Eq-7.2}
     g'(z) = \omega(z)h'(z)\;\mbox{and}\;	\omega(z) \in \mathcal{F}(\alpha), \;(z \in \mathbb{D})
\end{align}
where $\omega$ is the Möbius selfmap of $\mathbb{D}$ of the form $\omega(z) = \frac{z+\alpha/3}{1+z\alpha/3}$. The function $w$ has the series expansion
\begin{align}\label{Eq-7.3}
	\omega(z) = \alpha/3 + a_1 z + a_2 z^2 + \cdots\;\; (z \in \mathbb{D}, a_i \in \mathbb{C}, i = 1, 2, \ldots.)
\end{align}
If $f \in\mathcal{F}_{\mathcal{H}}(\alpha)$, then according to the form of $\omega(z) = (z+\alpha/3)/(1+z\alpha/3)$, and the relation $\omega=g'/h'$, we have $a_0=b_1=\alpha/3$,
\begin{align}\label{Eq-7.4}
	\frac{|r -\alpha/3|}{1 -r\alpha/3} \le |\omega(z)| \leq \frac{r +\alpha/3}{1 + r\alpha/3},
\end{align}
and
\begin{align}
	|a_n| \leq 1 - |a_0|^2 \;\;(n = 1, 2, \ldots),\;\;|\omega'(z)| \le \frac{1 - |\omega(z)|^2}{1 - |z|^2}\;\;(z \in \mathbb{D}).
\end{align}
The classical Bloch theorem asserts the existence of a positive constant $b$ such that for any holomorphic mapping $f$ of the unit disk $\mathbb{D}$, with the normalization $f'(0) = 1$, the image $f(\mathbb{D})$ contains a Schlicht disk of radius $b$. By Schlicht disk, we mean a disk which is the univalent image of some region in $\mathbb{D}$. The Bloch constant is defined as the \textquotedblleft best\textquotedblright{} such constant, that is supremum of such constants $b$. Chen et al.\cite{Chen-Gauthier-Hengartner-PAMS-2000} estimated Bloch constant for harmonic mappings.

A function $f \in \mathcal{H}ar(\mathbb{D})$ is called a harmonic Bloch mapping if and only if
\begin{align}
	\mathcal{B}_f= \sup_{z,\omega \in \mathbb{D}, z \ne \omega} \frac{|f(z) - f(\omega)|}{\mathcal{Q}(z,\omega)} < \infty,
\end{align}
where
\begin{align*}
	\mathcal{Q}(z,\omega)=\frac{1}{2} \log \left( \frac{1 + \left| \frac{z-\omega}{1-\overline{z}\omega} \right|}{1 - \left| \frac{z-\omega}{1-\overline{z}\omega} \right|} \right) = \text{arctanh} \left| \frac{z - \omega}{1 -\overline{z}\omega} \right|
\end{align*}
denotes the hyperbolic distance between $z$ and $\omega$ in $\mathbb{D}$, and $\mathcal{B}_f$ called the Bloch’s constantof $f$. In \cite{Colonna-IUMJ-1989} Colonna proved that
\begin{align}\label{Eq-7.7}
	\mathcal{B}_f = \sup_{z \in \mathbb{D}} (1 - |z|^2) \Lambda f
\end{align}
where
\begin{align*}
	\Lambda_f &= \Lambda_f (z) = \max_{0 \leq \theta \le 2\pi} \left| f_z(z) - e^{-2i\theta} {f_{\bar{z}}(z)} \right| = |f_z(z)| + |f_{\bar{z}}(z)|\\&=|h'(z)| + |g'(z)| = |h'(z)|(1 + |\omega(z)|).
\end{align*}
Moreover, the set of all harmonic Bloch mappings forms a complex Banach space with the norm $\|\cdot\|$ given by
\begin{align*}
	\|f\| = |f(0)| + \sup_{z\in\mathbb{D}} (1 - |z|^2)\Lambda_f (z).
\end{align*}
This definition agrees with the notion of the Bloch’s constant for analytic functions. Recently, many authors have studied Bloch’s constant for harmonic mappings. (see \cite{Chen-Gauthier-Hengartner-PAMS-2000,Kanas-Klimek-Smet-BM-2016,Liu-SCS-2009}).

Building upon the work presented in (see, \cite[Proposition 1]{Maharana-Prajapat-Srivastava-PNA-2017}), where a result was obtained for a distinct class of functions, we are motivated to derive a corresponding result for a relevant class of functions. This intermediate result then serves as a crucial step in the proof of our main theorem.
\begin{lem}\label{Lem-7.1}
	If $f\in\mathcal{F}(\alpha)$ of the form $f(z)=z+\sum_{n=2}^{\infty}a_nz^n$, then for $|z|=r<1$, we have
	\begin{enumerate}
		\item[(a)] $\bigg|\dfrac{zf^{\prime\prime}(z)}{f^{\prime}(z)}\bigg|\leq\dfrac{\alpha r}{1-r}$. The inequality is sharp and equality is attended for the function
		\begin{align}\label{Eq-7.8}
			g(z)=\int\frac{1}{(1-z)^\alpha}dz. 
		\end{align}\vspace{2mm}
		
		\item[(b)] $(1-z)^{-\alpha}\leq |f^{\prime}(z)|\leq (1+z)^{-\alpha}$. The inequality is sharp, with equality achieved by the function $g$ given by \eqref{Eq-7.8}.
	\end{enumerate}
\end{lem}
\begin{proof}[\bf Proof of Lemma \ref{Lem-7.1}]
	From the definition of $f\in\mathcal{F}(\alpha)$, we have
	\begin{align*}
		{\rm Re  }\left(1+\frac{zf^{\prime\prime}(z)}{f^{\prime}(z)}\right)\prec1+\frac{\alpha z}{1-z}
	\end{align*}
	which implies that
	\begin{align*}
		{\rm Re  }\left(\frac{zf^{\prime\prime}(z)}{f^{\prime}(z)}\right)\prec\frac{\alpha z}{1-z}=h(z).
	\end{align*}
	The result $(a)$ follows easily. To show the sharpness note that 
	\begin{align*}
		\frac{zg^{\prime\prime}(z)}{g^{\prime}(z)}=\frac{\alpha z}{1-z}.
	\end{align*}
	This equality, when $z=r, 0\leq r<1.$ To prove the result $(b)$ from the definition of the function $f\in\mathcal{F}(\alpha)$ and a well known subordination result in \cite{Suffridge-DMJ-1970}
	\begin{align*}
		f^{\prime}(z)&\prec\exp\left(\int_{0}^{z}\frac{h(t)}{t}dt\right)\\&=\exp\left(-\int_{0}^{z}\frac{-\alpha}{1-t}dt\right)=(1-t)^{-\alpha}.
	\end{align*}
	One can easily show the sharpness in $(b)$.
\end{proof}
This section focuses on finding bounds for the Bloch constant within the co-analytic part of functions in the class $\mathcal{F}_{\mathcal{H}}(\alpha)$.
\begin{thm}\label{Th-7.1}
	If $\alpha\in(0,2)$, and let $f\in\mathcal{F}_{\mathcal{H}}(\alpha)$. The bloch constant $\mathcal{B}_f$ is bounded by 
	\begin{align*}
		\mathcal{B}_f\leq\frac{(3+\alpha)(1-r_0^2)(1+r_0)^{1-\alpha}}{(3+\alpha r_0)}
	\end{align*}
	where $r_0$ is the unique root of the equation $6-\alpha-18r^2-(4\alpha+12)r^3-3\alpha r^4=0$ in the interval $(0,1)$.
\end{thm}
\begin{proof}[\bf Proof of Theorem \ref{Th-7.1}]
	Let $f=h+\overline{g}\in\mathcal{F}_{\mathcal{H}}(\alpha)$ and  $h\in\mathcal{F}(\alpha)$. Using Lemma \ref{Lem-7.1} along with \eqref{Eq-7.7} and \eqref{Eq-7.4}, we obtain
	\begin{align*}
		\mathcal{B}_f&=\sup_{z\in\mathbb{D}}(1-|z|^2)|h^{\prime}(z)|(1+|\omega(z)|)\\&\leq \sup_{z\in\mathbb{D}}(1-r^2)(1+r)^{-\alpha}\left(1+\frac{r+\alpha/3}{1+\alpha r/3}\right)=(3+\alpha)\sup_{0\leq r<1}\xi(r),
	\end{align*}
	where
	\begin{align*}
		\xi(r):=\frac{(1-r^2)(1+r)^{1-\alpha}}{3+\alpha r}.
	\end{align*}
	The derivative of $\xi(r)$ is equal to zero if $\eta(r) = 0$ for $r \in (0, 1)$, where
	\begin{align*}
		\eta(r)=3-4\alpha+(4\alpha-\alpha^2-9)r-(\alpha^2-2\alpha)r^2.
	\end{align*}
	We note that $\eta(0) = 3-4\alpha> 0$ holds for the value $\alpha<3/4$, and $\eta(1) = -2\alpha-7 < 0$, so that there exists a root $r_0 \in (0, 1)$ such that $\eta(r_0) = 0$. Now it suffices to prove that $r_0$ is unique. It is enough to prove that the derivative $\eta'(r) < 0$ for $r \in (0, 1)$ and $\alpha \in (0, 1)$. This holds by virtue of the inequality $\eta'(r) =4\alpha-\alpha^2-9-2(\alpha^2-2\alpha)r<0$ when $\alpha>2$. Hence
	\begin{align*}
		\sup_{0\leq r<1}\xi(r)=\frac{(1-r_0^2)(1+r_0)^{1-\alpha}}{(3+\alpha r_0)},
	\end{align*}
	where $r_0$ is unique root of $\eta(r) = 0$ for $r \in (0, 1)$. This proves the result.
\end{proof}
\section{\bf Geometric interpretation of the image domains for different values $\alpha\in [0, 3]$.}

\section{\bf Declarations}

\noindent\textbf{Conflict of interest:} The authors declare that there is no conflict  of interest regarding the publication of this paper.\vspace{1.2mm}

\noindent\textbf{Data availability statement:}  Data sharing not applicable to this article as no datasets were generated or analysed during the current study.\vspace{1.2mm}

\noindent {\bf Authors' contributions:} All the three  authors have equal contributions in preparation of the manuscript.


\begin{thebibliography}{99}
	
	\bibitem{Ahamed-Allu-Hossain-2025}{\sc M. B. Ahamed}, {\sc V. Allu}, {\sc R. Hossain},: Pre-Schwarzian and Schwarzian norm estimates, growth and distortion theorems for certain of analytic functions, \textit{Monatshefte für Mathematik}, (Accepted) 2025.
	
	\bibitem{Ahamed-Hossain-2024}{\sc M. B. Ahamed}, {\sc R. Hossain}: Pre-Schwarzian and Schwarzian norm estimates for class of Ozaki close-to-convex functions, \textit{arXiv:2412.18284v1}, (2024). 
	
	
	\bibitem{Ahamed-Hossain-Ahammed-2025}{\sc M. B. Ahamed}, {\sc R. Hossain} and {\sc S. Ahammed}: Schwarzian norm estimates for analytic functions associated with convex functions, \textit{arXiv:2506.19873v1}, ( Jun 2025). 
	
	\bibitem{Agrawal-Sahoo-IJPAM-2021}{\sc S. Agarwala}, {\sc S. K. Sahoo}: Nehari’s univalence criteria, pre-Schwarzian derivative and applications, \textit{Indian J. Pure Appl. Math.} \textbf{52(1)}, 193–204 (2021).
	
	\bibitem{aleman-2012} {\sc A. Aleman} and {\sc A. Constantin}, Harmonic maps and ideal fluid flows, {\it Arch. Ration. Mech. Anal.} {\bf 204} (2012), 479--513.
	
	
	 \bibitem{Ali-Pal-MJM-2023} {\sc M. F. Ali}, {\sc S. Pal},: Schwarzian norm estimates for some classes of analytic functions.  \textit{Mediterr. J. Math.} \textbf{20}(6), 294 (2023).
	 
	 
	\bibitem{Allu-Sharma-BSM-2024} {\sc V. Allu} and {\sc N. L. Sharma}, On logarithmic coefficients for classes of analytic functions associated with convex functions, \textit{Bull. Sci. Math.} \textbf{191} (2024), 103384.
	
	\bibitem{Avhadiev-Aksent\'ev-ETRMS-1975}{\sc F. G. Avhadiev}, {\sc L. A. Aksent\'ev}: Fundamental results on sufficient conditions for the univalence of analytic functions (Russian), \textit{Uspehi Mat. Nauk}, \textbf{30(4)}, 3-60, (1975). English translation in Russian	Math. Surveys \textbf{(30)}, 1–64 (1975). 
	

   \bibitem{Avkhadiev-Nasibullin-Shafigullin-RM-2016} {\sc F. G. Avkhadiev}, {\sc R. G. Nasibullin} and {\sc R. G. Shafigullin}, Becker type univalence conditions for harmonic mappings, \textit{Russian Math.}(Iz.VUZ) \textbf{60}(11), 69-73 (2016).

   \bibitem{Becker-JRAM-1972} {\sc J. Becker}, L\"ownersche differentialgleichung und quasikonform fortsetzbare schlichte functionen, \textit{J. Reine Angew. Math.} \textbf{255}, 23-43 (1972).


        \bibitem{Becker-Prmmerenke-AAS-1984} {\sc J. Becker-} and {\sc Ch. Prmmerenke}, Schlichtheitskriterien und Jordangebiete, \textit{J. Reine Angew. Math.} \textbf{354}, 74-94, (1984).
	
	    \bibitem{Chen-Gauthier-Hengartner-PAMS-2000} {\sc H. Chen}, {\sc P.M. Gauthier} and {\sc W. Hengartner}, Bloch constants theorem for planar harmonic mappings, \textit{Proc. Amer. Math. Soc.} \textbf{128}(2000), 3231–3240.
	    
	     
	    \bibitem{Chuaqui-Duren-Osgood-AAS-2011} {\sc M. Chuaqui}, {\sc P. Duren} and {\sc B. Osgood}, Schwarzian derivatives of convex mappings, \textit{Ann. Acad. Sci.
		Fenn. Math.} \textbf{36}(2), 449-460, (2011).
		
		\bibitem{Clunie- Sheil-Small-AAS-1984} {\sc J. Clunie} and {\sc T. Sheil-Small}, Harmonic univalent functions, \textit{Ann. Acad. Sci. Fenn. Math.} \textbf{9}(1984), 3-25.
		
		\bibitem{Chuaqui-Gervirtz-CVTA-2003} {\sc J. G. Chuaqui} and {\sc J. Gervirtz}, Quasidisks and the Noshiro-Warschawski criterion, \textit{Comp. Vari. Theo. Appl.} \textbf{48}, 967-985, (2003).
		
		\bibitem{Chuaqui-Harnandez-JMAA-2007} {\sc M. Chuaqui} and {\sc R. Harn\'andez}, Univalent harmonic mappings and linearly connected domains, \textit{J. Math. Anal. Appl.} \textbf{332}, 1189-1194, (2007).
			
			
	    \bibitem{Clunie-Sheil-AASFS-2007} {\sc J. G. Clunie} and {\sc T. Sheil-Small}, Harmonic univalent functions, \textit{Ann. Acad. Sci. Fenn. Ser. A I Math.} \textbf{9}, 3-25, (1984).
	    
	    \bibitem{constantin-2017} {\sc A. Constantin} and {\sc M. J. Martin}, A harmonic maps approach to fluid flows, {\it Math. Ann.} {\bf 369} (2017), 1-16.
	    
	    \bibitem{Colonna-IUMJ-1989} {\sc F. Colonna}, The Bloch constant of bounded harmonic mappings, {\it Indiana Univ. Math. J.} {\bf 38} (1989), 829–840.
	    
	    \bibitem{Duren-1983} {\sc P. T. Duren}, Univalent functions (Grundlehren der mathematischen Wissenschaften 259. Berlin,
	    Heidelberg, Tokyo), \textit{Springer-Verlag, New York}, (1983).
	    
	   \bibitem{Duren-Har-1983} {\sc P. T. Duren}, Harmonic mappings in the plane, \textit{Cambridge Univ. Press}, New York (2004).
	    
	    
	    \bibitem{Duren-Shapiro-Shields-DMJ-1966} {\sc P. T. Duren},  {\sc H. S. Shapiro} and {\sc A. L. Shields} Harmonic univalent functions, \textit{Ann. Acad. Sci. Fenn. Ser. A I Math.} \textbf{47}, 247-254, (1966).
	    
	    \bibitem{Garcia-Mashreghi-Ross-2018} {\sc S. R. Garcia},  {\sc J. Mashreghi} and {\sc W. T. Ross}, Finite Blaschke products and their connections, \textit{Springer,
	    	Cham} (2018).
	    
	    \bibitem{Hernández-Martín-JGA-2015} {\sc R. Hernández}, {\sc M. J. Martín},: Pre-Schwarzian and Schwarzian derivatives of harmonic mappings. \textit{J. Geom. Anal.} \textbf{25}(1), 64-91 (2015). 
	    

	    \bibitem{Hille-BAMS-1949} {\sc E. Hille}, Remarks on a paper be Zeev Nehari, \textit{Bull. Amer. Math. Soc.} \textbf{55}, 552-553 (1949).
	    
	    
	    \bibitem{Kanas-Klimek-Smet-BKMS-2014} {\sc S. Kanas},  {\sc D. Klimek-Smet}, Growth and coefficient estimates for uniformly locally univalent functions on
	    the unit disk, \textit{Bull. Korean Math. Soc.}, \textbf{51}(3)(2014), 803-812.
	    
	    \bibitem{Kanas-Klimek-Smet-BM-2016} {\sc S. Kanas},  {\sc D. Klimek-Smet}, Coefficient estimates and Bloch’s constant in some class of harmonic mappings, \textit{Bull. Malays. Math. Sci. Soc.}, \textbf{39}(2016), 741-750.
	    
	    \bibitem{Kanas-JMAA-2019} {\sc S. Kanas},  {\sc S. Maharana} and {\sc J. K. Prajapat}, Norm of the pre-Schwarzian derivative, Bloch’s constant and coefficient bounds in some classes of harmonic mappings, \textit{ J. Math.Anal.Appl.}, \textbf{474}(2019), 931–943.
	    
	    \bibitem{Kim-Sugawa-RMJ-2002} {\sc Y. C. Kim},  {\sc T. Sugawa}, Growth and coefficient estimates for uniformly locally univalent functions on
	    the unit disk, \textit{Rocky Mountain J. Math.} (2018).
	    	
	   	\bibitem{Lewy-BAMS-1936} {\sc H. Lewy}, On the non-vanishing of the Jacobian in certain one-to-one mappings, \textit{Bull. Amer. Math. Soc.} (1936).
	    		
	   \bibitem{Liulan-Ponnusamy-BM-2025} {\sc L. Li} and {\sc S. Ponnusamy}, New sufficient conditions for starlike and univalent functions, \textit{  Bull. Malays. Math. Soc.}, 48, 112 (2025).
	   
	   \bibitem{Liu-SCS-2009} {\sc M. Liu} and {\sc S. Ponnusamy}, Estimates on Bloch constants for planar harmonic mappings, \textit{ Sci. China Ser. A},\textbf{52}(2009), 87–93.
	   
	   \bibitem{Li-Pon-BMMSS-2025} {\sc L. Li} and {\sc S. Ponnusamy}, Geometric subfamily of functions convex in some
	   direction and Blaschke products, \textit{Bull. Malays. Math. Sci. Soc.} (2025) 48:112
	    		
	   \bibitem{Ma-Minda-1992} {\sc W. C. Ma} and {\sc D. Minda}, A unified treatment of some special classes of univalent functions, \textit{ Pro-ceedings of the Conference on Complex Analysis, Internat. Press, Cambridge, MA, USA}, (1992).
	   
	   \bibitem{Maharana-Prajapat-Srivastava-PNA-2017}, {\sc  S. Maharana}, {\sc J. K. Prajapat} and {\sc H. M. Srivastava}, The radius of convexity of partial sums of convex functions
	   in one direction, \textit{ Proc. Natl. Acad. Sci., India, Sect. A Phys. Sci.}, \textbf{87}(2)(2017), 215–219.
	   
	   \bibitem{Miller- Mocanu-JDE-1985} {\sc S. S. Miller} and {\sc P. T.  Mocanu}, Univalent solutions of Briot–Bouquet differential equations, \textit{ J. Differ. Equ.}, \textbf{56}, 297-309  (1985).
	    		
	    \bibitem{Nehari-BAMS-1949} {\sc Z. Nehari}, Some criteria of univalence, \textit{ Bull. Amer. Math. Soc.}, \textbf{55}(1949), 545-551.
	    			
	    \bibitem{Nehari-PAMS-1954} {\sc Z. Nehari}, Some criteria of univalence, \textit{Proc. Amer. Math. Soc.}, \textbf{5}(1954), 700-704.
	    	 
	    \bibitem{Nehari-IJM-1979} {\sc Z. Nehari}, Univalence criteria depending on the Schwarzian derivative, \textit{Illinois J. Math.}, \textbf{23}, 345-351 (1979).
	    	  
	   \bibitem{Obradovic-Ponnusamy-Wirths-SJM-2013} {\sc M. Obradovi\'c},  {\sc S. Ponnusamy} and {\sc K. -J. Wirths}, Coefficient characterizations and sections for some  univalent functions, \textit{Sib. Math. J.}, \textbf{54}(1), 679-696 (2013).
	   
	   \bibitem{Obradovic-Ponnusamy-Wirths-MM-2018} {\sc M. Obradovi\'c},  {\sc S. Ponnusamy} and {\sc K. -J. Wirths}, Logarithmic coefficients and a coefficient conjecture for
	   univalent functions, \textit{ Monatsh. Math.}, \textbf{185}(3), 489-501 (2018).
	   
	   \bibitem{Ozaki-TBDS-1941} {\sc S. Ozaki}, On the theory of multivalent functions II, \textit{ Sci. Rep. Tokyo Bunrika Daigaku. Sect. A}, \textbf{4}, 45-87 (1941).
	   
	   
	   \bibitem{Ozaki-SRTBD-1941}  {\sc S. Ozaki},: On the theory of multivalent functions. \textit{ Sci. Rep. Tokyo Bunrika Daigaku.} \textbf{4}, 455-486 (1941).
	   
	   \bibitem{Pommerenke-1975} {\sc Ch. Pommerenke}, Univalent functions. Vandenhoeck and Ruprecht, Göttingen, (1975).
	   
	   \bibitem{Ponnusamy-Rajasekran-SJM-1995} {\sc S. Ponnusamy} and {\sc S. Rajasekran}, New sufficient conditions for starlike and univalent functions.
	  , \textit{  Soochow J. Math.}, \textbf{21}, 193-201 (1995).
	   
	   \bibitem{Ponnusamy_sahoo-Sugawa-A-2014} {\sc S. Ponnusamy}, {\sc S. K. Sahoo} and {\sc T. Sugawa}, Radius problems associated with pre-Schwarzian and
	   Schwarzian derivatives, \textit{ Analysis (Munich)}, \textbf{34}, 163-171 (2014).
	   
	   \bibitem{Suffridge-DMJ-1970} {\sc T. J. Suffridge}, Some remarks on convex maps of the unit
	   disk, \textit{ Duke. Math. J.}, \textbf{37}(1970), 775–777.
	   
	   \bibitem{Umezawa-JMSJ-1952} {\sc T. Umezawa}, Analytic functions convex in one directionthe, \textit{ J. Math. Soc. Japan}, \textbf{4}, 194-202 (1952).
	   
	   
	   \bibitem{Wang-Li-Fan-MM-2024}  {\sc X. Wang}, {\sc H. Li}, {\sc J. Fan},: Pre-Schwarzian and Schwarzian norm estimates for subclasses of univalent functions. \textit{Monatsh. Math.} \textbf{205}(2024), 351-369.
\end{thebibliography}
\end{document}